\documentclass[12pt,a4paper]{article}
\usepackage{amsmath,amsthm,amssymb,amsfonts}
\usepackage{hyperref}
\usepackage{graphicx,subcaption}

\usepackage{enumitem}

\newcommand{\C}{{\mathbb{C}}}          
\newcommand{\R}{{\mathbb{R}}}          
\newcommand{\Sphere}{\mathbb{S}}

\newcommand{\fraku}{{\mathfrak{u}}}
\newcommand{\frakv}{{\mathfrak{v}}}

\newcommand{\XIS}{{\mathfrak{X}}}

\newcommand{\rr}{\rightarrow}
\newcommand{\lrr}{\longrightarrow}

\newcommand{\na}{{\nabla}}

\newcommand{\dx}{{\mathrm{d}}}

\newcommand{\papa}[2]{\frac{\partial#1}{\partial#2}}

\newcommand{\cz}{{\overline{z}}}
\newcommand{\cf}{{\overline{f}}}

\newcommand{\ch}{{\overline{h}}}
\newcommand{\cvarepsilon}{{\overline{\varepsilon}}}

\newcommand{\vol}{{\mathrm{vol}}}

\newcommand{\estrela}{{\boldsymbol{\star}}}

\newtheorem{teo}{Theorem}
\newtheorem{coro}{Corollary}
\newtheorem{prop}{Proposition}

\newenvironment{Rema}[1][Remark.]{\begin{trivlist}
\item[\hskip \labelsep {\bfseries #1}]}{\end{trivlist}}

\def\cyclic{\mathop{\kern0.9ex{{+}
\kern-2.2ex\raise-.28ex\hbox{\Large\hbox{$\circlearrowright$}}}}\limits}

\title{The volume of a unit vector field in 2 dimensions via calibrations}

\author{Rui Albuquerque}

\begin{document}


\maketitle


\begin{abstract}

We use the theory of calibrations to write the equation of a minimal volume vector field on a given Riemann surface.

\end{abstract}


\vspace*{1mm}

\vspace*{6mm}

\setcounter{section}{1}

\markright{\sl\hfill  Rui Albuquerque \hfill}

\vspace*{2mm}
\begin{center}
\begin{large}\textbf{{{1 -- Introduction}}}
\end{large}
\end{center}
\vspace*{2mm}

Gluck and Ziller used the theory of calibrations to prove that the minimal volume unit vector fields defined on the 3-dimensional sphere are the Hopf vector fields, \cite{GluckZiller}. Inspired by such work, we try to parallel those ideas on the setting of an oriented Riemannian 2-manifold.


In \cite{GluckZiller} an appropriate calibration 3-form $\varphi$ is found on the total space of the unit tangent sphere bundle $\pi:S\Sphere^3\lrr \Sphere^3$. The bundle sections are of course the unit vector fields on the base. Applying the theory of calibrations of Harvey and Lawson (\cite{HarLaw}), the corresponding embedded 3-dimensional $\mathrm{C}^2$ submanifolds calibrated by $\varphi$ are precisely the Hopf vector fields. They minimize volume globally in a unique homology class, namely the canonical class of the base $\Sphere^3$ which is included in $H_3(S\Sphere^3)$.


The question of minimality in dimension 2 has been raised before, but very little seems to be known. There are several important results e.g. in \cite{BorGil2010,BritoChaconJohnson,BritoGomesGoncalves,GilMedranoLlinaresFuster,Wieg}. To the best of our knowledge, a simple differential equation characterizing the 2-dimensional variational problem was missing.

Let $M$ denote a Riemann surface endowed with a unit norm $\mathrm{C}^2$ vector field $X$. By the same original definition in \cite{GluckZiller}, we have, cf. \cite{GilMedranoLlinaresFuster},
\begin{equation}  \label{Definition_volume}
\vol(X) =\vol(M,X^*g^S)=\int_M\sqrt{1+\|\na_{e_0}X\|^2+\|\na_{e_1}X\|^2}\,\vol_M
\end{equation}
where $g^S$ is the Sasaki metric on $SM$ and $e_0,e_1$ is \textit{any} local orthonormal
frame on $M$.

We denote by $\pi:SM\longrightarrow M$ the unit tangent sphere bundle of $M$, eventually with boundary. $SM$ is a Riemannian submanifold of metric contact type with contact 1-form $e^0$, this is, a contact manifold with compatible metric induced from $(TM,g^S)$ and contact structure induced from the geodesic spray.

When $M$ is oriented, there exists a natural differential system of 1-forms $e^0,e^1,e^2$ globally defined on $SM$ (which we like to see as the
simplest case of a fundamental differential system introduced in \cite{Alb2019b}). Let us recall at once the three structural equations of Cartan: $\dx e^0=e^{21}$, $\dx e^1=e^{02}$, $\dx e^2=K e^{10}$, where $K$ stands for the Gauss curvature of $M$.

It is clear how to find the global frame $e_0,e_1,e_2$ at each point $u\in SM$ such that $\pi(u)=x\in M$. The global vector field $e_0$ is the tautologial horizontal vector field, ie. the horizontal lift of $u\in T_xM$. In other words, $e_0$ is the geodesic spray vector field. Then $e_0,e_1$ is a well-defined direct orthonormal basis of horizontal vector fields; and $e_2$ is the vertical dual of $e_1$ tangent to the $S^1$ fibres.

In this article we start by characterizing a 2-form $\varphi$ on $SM$, clearly a linear combination of $e^{01},e^{20},e^{12}$, which defines the appropriate calibration for the study of unit vector fields on $M$.

We then establish the existence of $\varphi$ to that of a minimal vector field on $M$. Since $M$ must not satisfy any further restriction, our first theorem is also a local result. Indeed, we deduce an equation of a minimal volume vector field in any bounded domain: letting $A$ be essentially a $\C$-valued function given by the components of $\na_\cdot X$, we must have, in a conformal chart $z$ of $M$,
\begin{equation} \label{minimalvectorfieldcalibrated_intro}
 \papa{}{\cz}\frac{A}{\sqrt{1+|A|^2}}=0.
\end{equation}

Our second main result is the solution of \eqref{minimalvectorfieldcalibrated_intro} deduced over constant negative sectional curvature $K<0$.

For the reader to grasp the questions developed here below, we note the existence of a parallel vector field, clearly an absolute minima of the volume, starts as a local question. On the other hand, the theory of calibrations due to Harvey and Lawson applies to ma\-ni\-folds with boundary. So there is a path through geometry and topology here to pursue.

\vspace*{2mm}
\begin{center}
\begin{large}\textbf{{{2 -- Minimal volume over a surface}}}
\end{large}
\end{center}
\vspace*{2mm}

We start by recalling some general ideas in any dimension.

Let $(M,\langle\ ,\ \rangle)$ be an oriented Riemannian manifold of dimension $n+1$. Recall the
well-known metric and contact structure $e^0$ on the total space of
$\pi:SM\lrr M$. As usual, we let $e_0$ denote the geodesic spray, i.e. the unit norm horizontal vector field such that ${\dx\pi}_{u}(e_0)=u\in T_{\pi(u)}M,\
\forall u\in SM$.

Let us assume a calibration $\varphi$ is defined on $SM$.

Let $X\in\XIS_M$ be a class $\mathrm{C}^2$ unit norm vector field on $M$. As explained in \cite{HarLaw}, since $\varphi\leq\vol$ and since the $H_{n+1}(SM,\R)$ homology class of $X(M)$ is the same for all $X$, the minimal volume unit vector fields are those for which $\varphi=\vol$
when restricted to the submanifold ${X(M)}$. Indeed, recalling $\vol_X$ from \cite{GilMedranoLlinaresFuster,GluckZiller}, such unit vector fields are those for which $X^*\varphi=\vol_X$; corresponding to the so-called $\varphi$-submanifolds which are sections of $\pi:SM\lrr M$. Then the fundamental relation follows: for any unit $X'\in\XIS_M$,
\begin{equation}
   \int_M\vol_X =\int_{X(M)}\varphi=\int_{X'(M)}\varphi \leq  \int_M\vol_{X'}.
\end{equation}

The theory of calibrations holds for submanifolds with boundary of the
calibrated manifold. So we may well focus on a fixed open subset, a domain
$\Omega\subset M$ perhaps with non-empty boundary, and seek for an immersion 
$X:\Omega\rr SM$ giving a $\varphi$-submanifold. We remark that prescribing boundary values for $X$ on a compact $\partial\Omega$ implies that certain \textit{moment} conditions are
satisfied, cf. \cite[Eq.\,6.9]{HarLaw}.

Recalling a useful notation $\pi^*,\pi^\estrela$, for the horizontal, respectively vertical, canonical lift, we have the `horizontal plus vertical' decomposition $\dx
X(Y)=\pi^*Y+\pi^\star(\na_YX)$ in $TTM$. Also we may find local adapted frames
$e_0,e_1,\ldots,e_n,e_{1+n},\ldots,e_{2n}$, indeed, a local oriented orthonormal moving
frame on $SM$ with the $e_{i+n}$ vertical \textit{mirror} of the horizontal $e_i$,
$i=1,\ldots,n$.

$\pi^*X$ is the horizontal lift and thus $\pi^*X=e_0$ when we restrict to the submanifold $X(M)\subset SM$. This implies the pullback $X^*e^0=X^\flat$. The
horizontal $e_i$ project through $\dx\pi$ to a frame $e_i\in TM$ (same notation).
Hence, we may write
\begin{equation}
 \dx X(e_i)=e_i+\sum_{j=1}^n A_{ij}e_{j+n}
\end{equation}
for $i=0,1\ldots,n$, where $A_{ij}=\langle\na_{e_i}X,e_j\rangle$. Since $\|X\|=1$, $A_{i0}=0$.

We now suppose $M$ is a Riemann surface and $\pi:SM\lrr M$ is the unit circle tangent bundle. Let us search for the calibration $\varphi$.

As it is well-known, $SM$ is parallelizable. We have the global direct orthonormal frame $e_0,e_1,e_2$, with $e_2$ the vertical mirror of $e_1$. In particular $\pi^*\vol_M=e^{01}$.

The following formulas are well-known, cf. \cite{Alb2019b} and the references therein:
\begin{equation} \label{structuralexteriorderivatives}
 \dx e^0=e^2\wedge e^1,\qquad\dx e^1=e^0\wedge e^2, \qquad \dx e^2=K\,e^1\wedge e^0
\end{equation}
where $K=\langle R(e_0,e_1)e_1,e_0\rangle$ is the Gauss curvature. Notice $K$ 
is the pullback of a function on $M$ and it is not necessarily a constant.

Let us assume the abbreviation $b=\pi^*b$ for any given real function $b$
on $M$; this gives a function on $SM$ of course constant along the fibres.

Given $b_0,b_1,b_2\in\mathrm{C}^1_{M}(\R)$, we have a 2-form on $SM$:
\begin{equation}  \label{calibrationindim2}
 \varphi=b_2\,e^0\wedge e^1 + b_1\,e^2\wedge e^0 + b_0\,e^1\wedge e^2 .
 \end{equation} 
This is a 2-calibration if it has comass 1 and $\dx\varphi=0$. Recall from \cite{HarLaw} that comass 1 is defined by
\begin{equation}
 \sup\{\|\varphi\|^*_u:\ u\in SM\}=1
\end{equation}
where
\begin{equation}
 \|\varphi\|^*_u =\sup\bigl\{\langle\varphi_u,\xi\rangle:\ \xi\ \mbox{is a unit simple 2-vector at}\ u\bigr\}.
\end{equation}

\begin{prop} \label{prop_calibrationconditionindim2}
The 2-form $\varphi$ on $SM$ has comass 1 if and only if
 \begin{equation}  \label{calibrationconditionindim2}
  \sup\bigl\{{b_0}^2+{b_1}^2+{b_2}^2:\ x\in M\bigr\}=1.
 \end{equation}
The form $\varphi$ is closed if and only if the function $b_1+\sqrt{-1}b_0$ is holomorphic.
\end{prop}
\begin{proof}
For the first part, it is easy to deduce $\varphi(\fraku,\frakv)=\langle 
b_0e_0+b_1e_1+b_2e_2,\fraku\times\frakv\rangle$, for any $\fraku,\frakv$ tangent
to $SM$. We then recall that $\|\fraku\times\frakv\|=\|\fraku\wedge\frakv\|$. The definition of comass 1 together with Cauchy inequality yields $|(b_0,b_1,b_2)|\leq1$ and the requirement that the above supremum is 1. For the second part of the theorem, we note that $\dx b_i(e_2)=0, \forall i=0,1,2$, by construction. And therefore $\dx\varphi=0$ is
equivalent to the condition $\dx b_1(e_1)+\dx b_0(e_0)=0$. As the frame varies
along a single fibre we find Cauchy-Riemann equations. Hence the result.
\end{proof}
\begin{Rema}
There seems to be no advantage, later on, in considering general functions on $SM$; even if the equation $\dx b_2(e_2)+\dx b_1(e_1)+\dx b_0(e_0)=0$ sounds quite charmful. It is interesting to observe, by the way, that any two
functions $f,g$ on $M$, such that $\sup\{f^2+|\na g|^2\}=1$, define a 
calibration 2-form by $f\,e^0\wedge e^1 + \dx g\wedge e^2$.
\end{Rema}

Let us now seek for a calibration $\varphi$ on $SM$, intended for a new study on $M$.

Again let $X\in\XIS_M$ have unit norm and be defined over (a domain contained in) $M$. We then have a unique vector field $Y$ on the same domain such that $X,Y$ is a direct orthonormal frame.

The differential of the map $X$ is given by the
identities $\dx X(e_0)=e_0+A_{01}e_2$, $\dx X(e_1)=e_1+A_{11}e_2$, with usual 
notation $A_{ij}=\langle\na_{e_i}X,e_j\rangle$. In other words, abbreviating 
$A_{i1}=A_i$,
\begin{equation}
  X^*e^0=e^0,\qquad X^*e^1=e^1,\qquad X^*e^2=A_{0}e^0+A_{1}e^1  .
\end{equation}
Recalling definition \eqref{Definition_volume}, we find
\begin{equation}\label{Definition_volumeindim2}
\begin{split}
\vol_X &=\|\dx X(e_0)\wedge \dx X(e_1)\|\,e^{01}  \\
  &=\|e_0\wedge e_1+A_1e_0\wedge e_2+A_0e_2\wedge e_1\|\,e^{01} \\ 
  &=\sqrt{1+{A_{1}}^2+{A_{0}}^2}\,e^{01} .
\end{split}
\end{equation}
On the other hand,
\begin{equation} \label{pullbackvarphi}
 X^*\varphi=(-b_0A_{0}-b_1A_{1}+b_2)e^{01}.
\end{equation}
\begin{teo}  \label{teo_calibrationvsholomorphic}
 Suppose there exists a unit vector field $X$ on $M$ such that the $\C$-valued 
function $A=A_{1}+\sqrt{-1}A_{0}$ satisfies the following equation, in a 
conformal chart $z$ of $M$:
  \begin{equation}\label{minimalvectorfieldcalibrated}
  2(1+|A|^2)\papa{A}{\cz}-A\papa{|A|^2}{\cz}=0,
  \end{equation}
 corresponding to $A/\sqrt{1+|A|^2}$ being holomorphic. Then there exists a
calibration $\varphi$ on the total space of $SM$ for which $X$ is a
$\varphi$-submanifold. In particular, $X$ is a unit vector field on $M$ of 
minimal volume.
\end{teo}
\begin{proof}
 By Proposition \ref{prop_calibrationconditionindim2}, we search for a map
$\vec{b}=(b_0,b_1,b_2)$ from $M$ into the Euclidean ball of radius 1 and
having a limit value in the $\Sphere^2$ boundary. Let us also denote
$\vec{A}=(-A_0,-A_1,1)$.
 
 Now, by \eqref{Definition_volumeindim2} and \eqref{pullbackvarphi}, condition $X^*\varphi\leq\vol_X$ is equivalent to
 \[  \langle\vec{b},\vec{A}\rangle   \leq|\vec{A}| .   \]
 Since we wish equality and since $|\vec{b}|\leq1\leq|\vec{A}|$, there is a unique solution:
 \[  \vec{b}=\frac{\vec{A}}{|\vec{A}|} . \]
The corresponding $\varphi$ is globally defined, with the same domain as $X$. Finally, one must have $\varphi$ closed. Hence the function $A/\sqrt{1+|A|^2}$ must be holomorphic; and a
straightforward computation leads to \eqref{minimalvectorfieldcalibrated}.
\end{proof}

\begin{Rema}
 Seeing $A$ as $\na X$, one certainly finds inspiration for \eqref{minimalvectorfieldcalibrated} from the mi\-ni\-mal surface $u=u(x,y)$ graph equation in $\R^3$, due to Lagrange, cf. \cite[Eq. 1]{MeeksPerez}:
 \[ \mathrm{div}\Biggl(\frac{\na u}{\sqrt{1+|\na u|^2}}\Biggr)=0 . \]
\end{Rema}

\begin{coro}\label{corohyperbolic}
  Suppose $X$ is a solution of \eqref{minimalvectorfieldcalibrated}  such that the function $|A|$ is constant. Then $A$ is constant and the Riemann surface has constant sectional curvature $K=-|A|^2\leq0$. In particular,
  \begin{equation}
    \vol(X)=\sqrt{1-K}\,\vol(M) .
  \end{equation}
\end{coro}
\begin{proof}
 Let us use the notation $e_0=X$, $e_1=Y$ on $M$, as before. In general context, we have $\na_0e_0=A_{0}e_1$, $\na_1e_0=A_{1}e_1$ and so $\na_0e_1=-A_{0}e_0$, $\na_1e_1=-A_{1}e_0$. Hence $[e_0,e_1]=-A_{0}e_0-A_{1}e_1$ and then, cf. definition of $R$ in \cite{Alb2019b},
  \begin{align*}
  R(e_0,e_1)e_1 &= \na_{e_0}\na_{e_1}e_1-\na_{e_1}\na_{e_0}e_1-\na_{[e_0,e_1]}e_1  \\  
  & = -\na_0(A_{1}e_0)+\na_1(A_{0}e_0)+A_{0}\na_0e_1+A_{1}\na_1e_1 \\
   & =-\dx A_{1}(e_0)e_0-A_{1}A_{0}e_1+\dx A_{0}(e_1)e_0+ A_{0}A_{1}e_1-{A_{0}}^2e_0-{A_{1}}^2e_0 \\
    & =\dx A_{0}(e_1)e_0-\dx A_{1}(e_0)e_0-{A_{0}}^2e_0-{A_{1}}^2e_0.
  \end{align*}
 Now, if $|A|$ is constant, then from \eqref{minimalvectorfieldcalibrated} it follows that $A$ is holomorphic. Henceforth $A$ is constant. And thus $K=\langle R(e_0,e_1)e_1,e_0\rangle=-|A|^2$.
\end{proof}

Here follows a non-trivial complete example to which Corollary \ref{corohyperbolic} applies. It is the Lie group of affine transformations $M=\mathrm{Aff}(\R^2)$ with left invariant metric, together with any unit left invariant vector field $X$. It is easy to prove that $A$ is a constant.

$M$ is indeed a constant curvature hyperbolic surface, it is the 2-dimensional case of Special Example 1.7 from \cite{Milnor}, which is deduced there to be hyperbolic. Moreover, we know there are no other Lie groups of dimension 2 up to isometry with the same constant curvature $K<0$.

Equation \eqref{minimalvectorfieldcalibrated} proves quite hard to solve, be it for constant $K<0$ or $>0$. In the hyperbolic case, we cannot be sure about uniqueness of the solutions given by invariant theory.

\vspace*{2mm}
\begin{center}
\begin{large}\textbf{{3 -- In a conformal chart}}
\end{large}
\end{center}
\vspace*{2mm}

We seek further understanding of \eqref{minimalvectorfieldcalibrated} in general. Let us recall that a complex chart $z=x+iy$ corresponds with isothermal coordinates, ie. a chart such that the metric is given by $\lambda|\dx z|^2$ for some function $\lambda>0$.

A real vector field $X$ is given by $X=a\partial_x+b\partial_y=f\partial_z+\cf\partial_\cz$ where $f=a+ib$. If $Z=h\partial_z+\ch\partial_\cz$ is another vector field, then
\begin{equation}\label{metricconformal}
 \langle X,Z\rangle=(f\ch+\cf h)\frac{\lambda}{2}
\end{equation}
so that $\|X\|^2=f\cf\lambda$. We have $Y=if\partial_z-i\cf\partial_\cz 
=\overline{Y}$.

Recall the Levi-Civita connection, a real operator, is given by 
$\na_z\partial_z=\Gamma\partial_z$ where 
$\Gamma=\frac{1}{\lambda}\papa{\lambda}{z}$, 
$\na_z\partial_\cz=\na_\cz\partial_z=0$, 
$\na_\cz\partial_\cz=\frac{1}{\lambda}\papa{\lambda}{\cz}\partial_\cz$. In particular we
have $R(\partial_z,\partial_\cz)\partial_z=-\papa{\Gamma}{\cz}\partial_z$ and 
hence
\begin{equation}
K=\frac{\langle R(\partial_z,\partial_\cz)\partial_z,\partial_\cz\rangle}{\langle\partial_z,\partial_\cz\rangle^2}=-\frac{2}{\lambda}\papa{\Gamma}{\cz}=-\frac{2}{\lambda}\papa{^2\log\lambda}{z\partial\cz}.
\end{equation}
Therefore $\na_XX=\varepsilon_0\partial_z+\cvarepsilon_0\partial_\cz$ and $\na_YX=i\varepsilon_1\partial_z-i\cvarepsilon_1\partial_\cz$ where
\begin{equation}
 \varepsilon_0=ff'_z+\frac{f^2}{\lambda}\lambda'_z+\cf f'_\cz \qquad\mbox{and}\qquad \varepsilon_1=ff'_z+\frac{f^2}{\lambda}\lambda'_z-\cf f'_\cz .
\end{equation}
We have $A_1=\langle \na_YX,Y\rangle=(\varepsilon_1\cf+\cvarepsilon_1f)\frac{\lambda}{2}$ and $A_0=\langle \na_XX,Y\rangle=(-i\varepsilon_0\cf+i\cvarepsilon_0f)\frac{\lambda}{2}$. Now for a unit vector we have the identity $f'_z\cf\lambda+f\cf'_z\lambda+f\cf\lambda'_z=0$ and its conjugate. This yields $A_0=i\lambda(f^2\cf'_z-\cf^2f'_\cz)$ and $A_1=-\lambda(f^2\cf'_z+\cf^2f'_\cz)$, finally giving a simple and noteworthy result.
\begin{prop}
 $A=-2\lambda f^2\cf'_z =2(\Gamma f+f'_z)$.
\end{prop}
We note that $|A|=2|\cf'_z|$ and that a holomorphic unit vector field is just a parallel vector field.

Finding $f$ from equation \eqref{minimalvectorfieldcalibrated} in Theorem \ref{teo_calibrationvsholomorphic} proves quite difficult even for the trivial non-flat metrics.

On the round $\Sphere^2$ punctured at two antipodal points, it is stated and proved in \cite{BritoChaconJohnson} that a minimum of $\vol(X)$ is attained: a solution $X_0$ is given, for instance, by the directed meridians unit tangent vector field, invariant by parallel transport between poles. However, this solution does not solve \textit{our} equation --- which is not surprising!, for we have found vector fields with even less volume than $X_0$ in a smaller open region of $\Sphere^2$. Such result will be shown in a proper article.

\begin{small}

\ \\
\textsc{R. Albuquerque}\ \ \ \textbar\ \ \ 
{\texttt{rpa@uevora.pt}}\\
Centro de Investiga\c c\~ao em Mate\-m\'a\-ti\-ca e Aplica\c c\~oes\\
Rua Rom\~ao Ramalho, 59, 671-7000 \'Evora, Portugal\\
The research leading to these results has received funding from Funda\c c\~ao para a Ci\^encia e a Tecnologia. Project Ref. UIDB/04674/2020.

\end{small}

\end{document}